\documentclass[a4paper,reqno,11pt]{amsart}
\usepackage[a4paper,margin=2.5cm]{geometry}
\usepackage[T1]{fontenc}
\usepackage[utf8]{inputenc}
\usepackage[english]{babel}
\usepackage{amssymb,amsmath,amsthm,mathtools}
\usepackage[dvipsnames,svgnames,table]{xcolor}
\usepackage{graphicx}
\usepackage{caption,subcaption}
\usepackage{lmodern}
\usepackage[foot]{amsaddr}
\usepackage[shortlabels]{enumitem}
\usepackage[unicode=true]{hyperref}
\usepackage[capitalise,noabbrev]{cleveref}
\usepackage{thmtools,thm-restate}
\usepackage{comment}
\usepackage{mdframed}
\usepackage{algorithmicx}
\usepackage[noend]{algpseudocode}

\DeclarePairedDelimiter\set{\{}{\}}

\DeclarePairedDelimiter\ceil{\lceil}{\rceil}

\newcommand{\defin}[1]{\emph{\textcolor{ForestGreen}{#1}}}
\newcommand{\defmath}[1]{\defin{\text{$#1$}}}

\newcommand{\Oh}{\mathcal{O}}

\newcommand{\chC}{\chi_{\mathrm{cen}}}
\newcommand{\chL}{\chi_{\mathrm{lin}}}

\let\le\leqslant
\let\ge\geqslant
\let\leq\leqslant
\let\geq\geqslant

\let\epsilon\varepsilon

\let\setminus\backslash

\crefformat{equation}{#2(#1)#3}
\crefformat{subsection}{Subsection #2#1#3}

\makeatletter
\def\thm@space@setup{
  \thm@preskip=2mm
  \thm@postskip=0mm
}
\makeatother

\algdef{SE}[DOWHILE]{Do}{doWhile}{\algorithmicdo}[1]{\algorithmicwhile\ #1}

\mdfdefinestyle{dontsplit}{
  hidealllines=true,
  nobreak=true,
  leftmargin=0pt,
  rightmargin=0pt,
  innerleftmargin=0pt,
  innerrightmargin=0pt,
}

\newmdtheoremenv[style=dontsplit]{theorem}{Theorem}
\newmdtheoremenv[style=dontsplit]{lemma}[theorem]{Lemma}
\newmdtheoremenv[style=dontsplit]{observation}[theorem]{Observation}
\newmdtheoremenv[style=dontsplit]{proposition}[theorem]{Proposition}
\newmdtheoremenv[style=dontsplit]{question}[theorem]{Question} 
\newmdtheoremenv[style=dontsplit]{corollary}[theorem]{Corollary} 
\newmdtheoremenv[style=dontsplit]{problem}[theorem]{Problem}
\newmdtheoremenv[style=dontsplit]{conjecture}[theorem]{Conjecture}

\newtheorem*{problem*}{Problem}
\newtheorem*{conjecture*}{Conjecture}
\newtheorem*{remark*}{Remark}
\newtheorem*{question*}{Question} 

\theoremstyle{remark}
\newmdtheoremenv[style=dontsplit]{claim}[theorem]{Claim}

\crefname{claim}{Claim}{Claims}

\newtheorem*{claim*}{Claim}

\graphicspath{{figs/}}

\setlist[enumerate,1]{label=\textup{(\roman*)}}

\makeatletter
\newcommand{\myitem}[1]{%
\item[#1]\protected@edef\@currentlabel{#1}%
}
\makeatother

\hypersetup{
    colorlinks,
    linkcolor={RoyalBlue},
    citecolor={RubineRed},
    urlcolor={blue!80!black},
    pdftitle={Superlinear Separation Between Linear and Centered Colorings},
    pdfauthor={Jędrzej Hodor and Piotr Micek}
}

\let\phi\varphi

\begin{document}

\title[Superlinear Separation Between Linear and Centered Colorings]
{Superlinear Separation Between Linear and Centered Colorings}

\author[J.~Hodor]{J\k{e}drzej Hodor}
\address[J.~Hodor]{Theoretical Computer Science Department,
Faculty of Mathematics and Computer Science and Doctoral School of Exact and Natural Sciences,
Jagiellonian University, Krak\'ow, Poland}
\email{\href{mailto:jedrzej.hodor@gmail.com}{jedrzej.hodor@gmail.com}}

\author[P.~Micek]{Piotr Micek}
\address[P.~Micek]{Theoretical Computer Science Department, Faculty of Mathematics and Computer Science, Jagiellonian University, Kraków, Poland.}
\email{\href{mailto:piotr.micek@uj.edu.pl}{piotr.micek@uj.edu.pl}}
\thanks{The authors were partially supported by the National Science Centre, Poland,  under grant UMO-2023/05/Y/ST6/00079 within the WEAVE-UNISONO program.}

\begin{abstract}



A vertex-coloring of a graph is centered if every connected subgraph has a vertex with a unique color. 
A vertex-coloring of a graph is linear if every path in the graph has a vertex with a unique color. 
Let $\chC(G)$ and $\chL(G)$ be the minimum number of colors in a centered (resp.\ linear) coloring of $G$. 
We present a family of graphs witnessing that if $f$ is a nondecreasing function such that $\chC(G) \leq f(\chL(G))$ for every graph $G$, then
$f(k) = \Omega(k^2 / \log k)$.
The construction was found by OpenAI's GPT-5.6 Sol Pro.

\end{abstract}

\maketitle

\section{Introduction}\label{sec:introduction}

For a coloring $\varphi$ of a graph $G$, a vertex $v$ of a subgraph $H$ of $G$ is a \defin{$\varphi$-center} of $H$ if the color $\varphi(v)$ appears exactly once in $H$.\footnote{All graph colorings in this paper are vertex-colorings.}
A coloring $\varphi$ of a graph $G$ is \defin{centered} if every connected subgraph of $G$ has a $\varphi$-center.
The \defin{centered chromatic number} of a graph $G$ is the minimum number of colors in a centered coloring of $G$; we denote this value by \defin{$\chC(G)$}. 
The centered chromatic number is a fundamental parameter in structural graph theory; it is also known as treedepth, elimination height, or unique maximum chromatic number.

Motivated by algorithmic applications, Cheilaris and Tóth~\cite{CT11} introduced the following relaxation of centered colorings.\footnote{Cheilaris and Tóth use the name \emph{conflict-free coloring with respect to paths}. However, for consistency with more recent work, we use the name \emph{linear coloring}.}
A coloring $\varphi$ of a graph $G$ is \defin{linear} if every path in $G$ has a $\varphi$-center.
The \defin{linear chromatic number} of a graph $G$ is the minimum number of colors in a linear coloring of $G$; we denote this value by \defin{$\chL(G)$}.

It follows directly from the definitions that $\chL(G) \leq \chC(G)$ for every graph $G$.
Conversely, if $\chL(G) \leq k$, then $G$ contains no path on $2^k$ vertices as a subgraph, which immediately implies $\chC(G) < 2^k$.
Kun, O’Brien, Pilipczuk, and Sullivan~\cite{Kun2020} substantially improved this exponential bound by showing that $\chC(G)$ is bounded polynomially in $\chL(G)$.
The best-known bound, $\chC(G) \leq \widetilde{\Oh}(\chL(G)^{10})$, follows from results of Czerwi{\'n}ski, Nadara, and Pilipczuk~\cite{CNP}, and Bose, Dujmović, Houdrouge, Javarsineh, and Morin~\cite{Bose}.
Kun et al.~\cite{Kun2020} conjectured the much stronger inequality $\chC(G) \leq 2 \chL(G)$ for every graph $G$.
We disprove this conjecture.

\begin{theorem}\label{thm:main}
There exists a family of graphs $(G_k)_{k=1}^\infty$ and an increasing unbounded sequence $(x_k)_{k=1}^\infty$ of positive integers such that $\lim_{k \rightarrow \infty} x_{k+1} / x_k = 1$,
\[
\chL(G_k)\leq x_k,\ \ \text{ and }\ \  \chC(G_k)\geq\left(\frac{1}{2}-o(1)\right)\frac{x_k^2}{\log x_k}.
\]
\end{theorem}

This implies that if $f$ is a nondecreasing function such that for every graph $G$, we have $\chC(G) \leq f(\chL(G))$, then
$f(k) = \Omega(k^2 / \log k)$.

\section{Construction}\label{sec:preliminaries}
Throughout, \defin{$\log(\cdot)$} denotes the base-$2$ logarithm and \defin{$\ln(\cdot)$} denotes the natural logarithm.
For a positive integer $p$, write $\defmath{[p]}=\{1,\ldots,p\}$, and let $\defmath{[0]} = \emptyset$.
All graphs are finite, simple, and undirected.
Connected graphs and paths are assumed to be nonnull.
Let $(a_k)_{k=1}^\infty$ and $(b_k)_{k=1}^\infty$ be two real-valued sequences.
We write \defin{$a_k \sim b_k$} when $\lim_{k\rightarrow \infty} a_k / b_k = 1$.
We also write \defin{$a_k = o(b_k)$} when $\lim_{k\rightarrow \infty} a_k / b_k = 0$.
The symbol \defin{$1$} denotes the constant sequence equal to $1$, hence the notation \defin{$o(1)$}.

Let $F$ be a graph and let $p$ be a nonnegative integer.
We construct a graph \defin{$L(F,p)$} as follows.
Let $h = \ceil{\log(p+1)}$ and fix an arbitrary enumeration $B_1,\dots,B_{2^h}$ of all subsets of $[h]$. 
Note that $p+1 \leq 2^h$. 
We start the construction with a complete graph on the vertex set $\{v_1,\dots,v_{p+1}\}$.
For every $i \in [p+1]$, attach to $v_i$ a path $P_i$ on $|B_i|+1$ vertices, whose other endpoint is called $u_i$.
All vertices of $P_i - v_i$ are new.
Note that $u_i = v_i$ when $B_i = \emptyset$.
Fix $p+1$ vertex-disjoint copies $F_1,\dots,F_{p+1}$ of $F$ on fresh vertices.
Finally, for every $i \in [p+1]$, we add an edge between $u_i$ and each vertex of $F_i$. 
This completes the description of $L(F,p)$. 
\begin{lemma}\label{lem:lin}
    Let $F$ be a graph, let $p$ be a nonnegative integer with 
    $\chL(F)\leq p$. Then  
    \[\chL(L(F,p)) \leq p + 1 + \ceil{\log (p+1)}.\]
\end{lemma}
\begin{proof}
    Let $h = \ceil{\log(p+1)}$.
    For every $i \in [p+1]$, let $\varphi_i$ be a linear coloring of $F_i$ taking values in $[p+1] \setminus\{i\}$.
    We construct a coloring $\psi$ of $L(F,p)$ using colors from $[p+1] \sqcup \{\alpha_1,\dots,\alpha_h\}$ as follows.
    Let $i \in [p+1]$.
    Recall that $|V(P_i)\setminus\set{v_i}| = |B_i|$. 
    We arbitrarily enumerate the vertices of $P_i - v_i$ as $\set{z_{i,b} : b \in B_i}$. 
    Then, we define
    \[
    \psi(v) =\begin{cases}
    i&\textrm{if $v=v_i$ and $i\in[p+1]$,}\\
    \phi_i(v)&\textrm{if $v\in V(F_i)$ and $i\in[p+1]$,}\\
    \alpha_b&\textrm{if $v\in V(P_i)\setminus \set{v_i}$, $v=z_{i,b}$, $i\in[p+1]$, and $b\in B_i$.}
    \end{cases}
    \]
    
    We claim that $\psi$ is a linear coloring of $L(F,p)$. 
    Thus, we need to show that each path in $L(F,p)$ has a $\psi$-center.
    By construction, every path in $L(F,p)$ intersects the vertex set of at most two graphs among $F_1,\ldots,F_{p+1}$. 
    Let $P$ be a path in $L(F,p)$.
    Let $\gamma = |\{i \in [p+1] : V(P) \cap V(F_i) \neq \emptyset\}|$.
    Thus, $\gamma \in\set{0,1,2}$.
    
    First, assume that $\gamma = 0$.
    If $P$ is a subgraph of $P_i$ for some $i \in [p+1]$, then $\psi$ is an injective coloring on $V(P)$ and hence every vertex of $P$ is a $\psi$-center.
    Otherwise, $V(P)$ intersects $\{v_1,\dots,v_{p+1}\}$ and any vertex in this intersection is a $\psi$-center of $P$.

    Next, assume that $\gamma = 1$ and let $i \in [p+1]$ be such that $P$ intersects $F_i$. 
    If $P$ is contained in $F_i$, then $P$ has a $\psi$-center as $\varphi_i$ is a linear coloring of $F$.
    Otherwise, $P$ contains a vertex of $P_i$. If $P$ is contained in $L(F,p)[V(F_i) \cup V(P_i)]$, then any vertex in $V(P)\cap V(P_i)$ is a $\psi$-center of $P$. 
    Finally, if $P$ contains a vertex in $\{v_1,\dots,v_{p+1}\}$, then, since $P$ is connected, $v_i \in V(P)$, and $v_i$ is a $\psi$-center of $P$.
    This exhausts the possibilities in the case $\gamma = 1$.

    Finally, assume that $\gamma = 2$ and let $i,j \in [p+1]$ be distinct indices such that $P$ intersects $F_i$ and $F_j$.
    In this case, both $P_i$ and $P_j$ are subpaths of $P$.
    Note that, by construction, $P$ contains no other vertex with a color from $\{\alpha_1,\dots,\alpha_h\}$.
    Since $i \neq j$, we also have $B_i \neq B_j$.
    Let $b$ be an element of the symmetric difference of $B_i$ and $B_j$.
    Then $\alpha_b$ is a $\psi$-center of $P$.
    This completes the proof that $\psi$ is a linear coloring of $L(F,p)$.
\end{proof}

Let $F$ be a graph and let $p$ be a nonnegative integer.
We construct a graph \defin{$A(F,p)$} as follows.
We start the construction with a complete graph on the vertex set $\{v_1,\dots,v_{p+1}\}$.
Fix $p+1$ vertex-disjoint copies $F_1,\dots,F_{p+1}$ of $F$ on fresh vertices.
For every $i \in [p+1]$, we add an edge between $v_i$ and each vertex of $F_i$. 
This completes the description of $A(F,p)$. 
Note that $A(F,p)$ can be obtained from $L(F,p)$ by contracting all paths $P_i$.
Since centered chromatic number is minor-monotone,
\begin{equation}\label{eq:minor}
    \chC(A(F,p)) \leq \chC(L(F,p)).
\end{equation}

\begin{lemma}\label{lem:cen}
    Let $F$ be a graph, let $p$ be a nonnegative integer. Then  
    \[\chC(A(F,p)) \geq p + 1 + \chC(F).\]
\end{lemma}
\begin{proof}
    We proceed by induction on $p$. 
    If $p = 0$, then $A(F,p)$ is the graph $F$ with an additional universal vertex, and therefore $\chC(A(F,p)) \geq 1 + \chC(F)$.
    Thus, assume that $p \geq 1$.
    Let $\varphi$ be a centered coloring of $A(F,p)$ using exactly $\chC(A(F,p))$ colors.
    It follows that $A(F,p)$ has a $\varphi$-center, say $v$.
    Let $i \in [p+1]$ be such that $v \in V(F_i) \cup \{v_i\}$.
    Note that the graph $A(F,p) - (V(F_i) \cup \{v_i\})$ is isomorphic to $A(F,p-1)$.
    The coloring $\varphi$ does not use the color of $v$ on the vertices of $A(F,p-1)$.
    By induction, $\varphi$ must use at least $p + \chC(F)$ colors on the vertices of $A(F,p-1)$.
    Therefore, $\varphi$ uses at least $p + 1 + \chC(F)$ colors, as desired.
\end{proof}

We now construct a family witnessing \cref{thm:main}.
First, we need two recursively defined sequences $(\ell_k)_{k=1}^\infty$ and $(c_k)_{k=1}^\infty$:
\begin{align}
    \ell_1=1,& \label{eq:k-rec}
    \ \ 
    \ell_{k+1}=\ell_k+1+\ceil*{\log(\ell_k+1)},\\
    c_1=1,& \label{eq:d-rec}
    \ \
    c_{k+1}=c_k+\ell_k+1.
\end{align}
Now, start with $G_1=K_1$, and recursively, having constructed a graph $G_k$, set
\[
G_{k+1}=L(G_k,\ell_k).
\]

We have $\chL(G_1) = \chC(G_1) = 1$.
Inductively, \cref{lem:lin} and \cref{lem:cen} together with \cref{eq:minor} imply that for every positive integer $k$, we have
\begin{equation}\label{eq:parameter-bounds}
\chL(G_k)\le \ell_k
\ \ \text{ and } \ \
\chC(G_k)\ge c_k.
\end{equation}
We will show the following technical lemma, which, as we first prove, implies \cref{thm:main}.

\begin{lemma}\label{lem:asymptotics}
The sequences $(\ell_k)_{k\ge1}$ and $(c_k)_{k\ge1}$ defined by the recurrences \cref{eq:k-rec} and \cref{eq:d-rec} satisfy
\[
c_k\sim \frac{\ell_k^2}{2\log \ell_k}.
\]
\end{lemma}

\begin{proof}[Proof of \Cref{thm:main}]
    We set $(x_k)_{k=1}^\infty = (\ell_k)_{k=1}^\infty$.
    By \cref{lem:asymptotics}, $\lim_{k \rightarrow \infty} c_k / (x_k^2 / 2\log x_k) = 1$.
    It follows from~\cref{eq:parameter-bounds} that $\chC(G_k)\geq c_k \geq (1/2 - o(1))x_k^2 / \log x_k$.
    This completes the proof of the theorem.
\end{proof}

\begin{proof}[Proof of \cref{lem:asymptotics}]
Let
\[
\delta_k=\ell_{k+1}-\ell_k=1+\ceil*{\log(\ell_k+1)}.
\]
By \cref{eq:k-rec}, $\delta_k\sim\log \ell_k$ and $\delta_k=o(\ell_k)$.
For each real number $x > 1$, we set
\[
g(x)=\frac{x^2}{2\log x}.
\]
Its derivative is
\[
g'(x)=\frac{x}{\log x}\left(1-\frac{1}{2\ln x}\right).
\]
By the mean value theorem, for every $k \geq 2$, there exists $\xi_k\in(\ell_k,\ell_{k+1})$ such that $g(\ell_{k+1})-g(\ell_k)=g'(\xi_k)\delta_k$.
Since $\delta_k=o(\ell_k)$, we have $\xi_k \sim \ell_k$ and $\log\xi_k \sim \log \ell_k$.
Also, $(1 - (1/ 2\ln \xi_k)) \sim 1$.
Consequently,
\begin{align*}
    g(\ell_{k+1})-g(\ell_k)=g'(\xi_k)\delta_k &= \frac{\xi_k}{\log \xi_k}\cdot \left(1-\frac{1}{2\ln \xi_k}\right) \cdot \delta_k 
    \\ &\sim \frac{\ell_k}{\log \ell_k} \cdot 1 \cdot \log \ell_k \sim \ell_k.
\end{align*}
On the other hand, \cref{eq:d-rec} gives $c_{k+1}-c_k=\ell_k+1\sim \ell_k$, and therefore, $g(\ell_{k+1})-g(\ell_k) \sim c_{k+1}-c_k$.
This yields $g(\ell_k) \sim c_k$ and completes the proof.
\end{proof}

\textbf{Remark. }
Note that all graphs in $(G_k)_{k=1}^\infty$ are chordal. 
Hilaire, Krnc, Milanič, and Raymond~\cite{slovenians} proved that 
$\chC(G)=\Oh(\chL(G)^2)$ for all chordal graphs $G$. 
Thus, we obtain an almost tight bound on the optimal binding function for this class.

\section*{Statement of AI use}
The construction was found by OpenAI's GPT-5.6 Sol Pro. 
The authors take responsibility for the
mathematical correctness of the presented arguments.

\bibliographystyle{plain}
\bibliography{bibliography}
\end{document}